\documentclass[reqno,12pt]{amsart}

\usepackage{amsaddr}
\usepackage{amssymb,amsmath,amsthm,xcolor,enumerate,hyperref,cleveref,mathrsfs}
\usepackage{pgf,tikz}
\usetikzlibrary{arrows}

\allowdisplaybreaks

\newtheorem{thrm}{Theorem}[section]
\newtheorem{cor}[thrm]{Corollary}
\newtheorem{lem}[thrm]{Lemma}
\newtheorem{prop}[thrm]{Proposition}

\theoremstyle{definition}
\newtheorem{defn}[thrm]{Definition}
\newtheorem{exm}[thrm]{Example}

\newtheorem{rem}[thrm]{Remark}

\crefrangeformat{equation}{#3(#1)#4--#5(#2)#6}

\crefname{thrm}{Theorem}{Theorems}
\crefname{lem}{Lemma}{Lemmas}
\crefname{cor}{Corollary}{Corollaries}
\crefname{prop}{Proposition}{Propositions}
\crefname{defn}{Definition}{Definitions}
\crefname{exm}{Example}{Examples}
\crefname{rem}{Remark}{Remarks}
\crefname{section}{Section}{Sections}
\crefname{equation}{\unskip}{\unskip}
\crefname{enumi}{\unskip}{\unskip}

\DeclareMathOperator{\spn}{span}

\newcommand{\impl}{\Rightarrow}

\newcommand{\lb}{\lambda}
\newcommand{\sg}{\sigma}

\newcommand{\sst}{\subseteq}

\newcommand{\tl}{\tilde}

\begin{document}

	\noindent{\Large  
		Poisson Structures on Finitary Incidence Algebras}\footnote{
		The work is supported by  CNPq     404649/2018-1 and 	302980/2019-9.}

	\
	
	{\bf
		Ivan Kaygorodov$^{a}$ \& Mykola Khrypchenko$^{b}$  \\
		
		\medskip
		
		\medskip
	}
	
	{\tiny

		\smallskip

		$^{a}$ CMCC, Universidade Federal do ABC, Santo Andr\'e, Brazil
		
		\smallskip

		$^{b}$ Departamento de Matem\'atica, Universidade Federal de Santa Catarina,     Brazil

		\
		
		\smallskip
		
		\medskip
		
		E-mail addresses:

		\smallskip

		Ivan Kaygorodov (kaygorodov.ivan@gmail.com)

		\smallskip
		
		Mykola Khrypchenko  
		(nskhripchenko@gmail.com)
		
	}
	
	\ 
	
	\ 
	
	\medskip

	\ 
	
	\noindent {\bf Abstract:}
		We give a full description of the Poisson structures on the finitary incidence algebra $FI(P,R)$ of an arbitrary poset $P$ over a commutative unital ring $R$.
		
	\
	
	
	\ 
	

	\section*{Introduction}\label{intro}
	
	The notion of a Poisson bracket has its origin in the works of S.D. Poisson on celestial mechanics of the beginning of the XIX century. Since then, Poisson algebras
	have appeared in several areas of mathematics, such as:  
	Poisson manifolds~\cite{L1}, 
	algebraic geometry~\cite{BLLM,GK04}, 
	noncommutative geometry \cite{ber08},
	operads~\cite{MR},
	quantization theory~\cite{Hue90,Kon03}, 
	quantum
	groups~\cite{Dr87}, 
	and classical and quantum mechanics.
	The study of Poisson algebras also led to other algebraic structures, such as 
	generic Poisson algebras,
	algebras of Jordan brackets and generalized Poisson algebras,
	Gerstenhaber algebras,
	Novikov-Poisson algebras,
	Malcev-Poisson-Jordan algebras, 
	transposed Poisson algebras,
	$n$-ary Poisson algebras, etc.
	
	Poisson algebras have been used to prove the Nagata conjecture about wild automorphisms of the polynomial ring  with three generators \cite{shu04}
	and to describe simple noncommutative Jordan algebras and superalgebras \cite{ps19}.
	The systematic study of noncommutative Poisson algebra structures began in the paper of Kubo \cite{kubo96}.
	He obtained a description of all the Poisson structures on the full and upper triangular matrix algebras, which was then generalized to prime associative algebras in \cite{fale98}.
	Namely, it was proved in \cite{fale98} that any Poisson bracket on a prime noncommutative associative algebra is the commutator bracket multiplied by an element from the extended centroid of the algebra.
	On the other hand, in his next paper, Kubo studied 
	noncommutative Poisson algebra structures on affine Kac-Moody algebras \cite{kubo98}.
	The investigation of Poisson structures on associative algebras continued in some papers of
	Yao,   Ye  and Zhang \cite{yyz07}; Mroczyńska, Jaworska-Pastuszak,   and  Pogorzały  \cite{jpp20,mp18},
	where Poisson structures on finite-dimensional path algebras
	and on canonical algebras were studied.
%	On the other hand, commutative Poisson structures were described on finite-dimensional semisimple Lie algebras \cite{bb14} and on oscillator Lie algebras \cite{abbbs20}.
	Crawley-Boevey  introduced a noncommutative Poisson structure, called an $H_0$-Poisson structure, on the $0$-th cyclic homology of a  noncommutative associative algebra \cite{cb11}
	and showed that an $H_0$-Poisson structure can be induced on the affine moduli space of (semisimple) representations of an associative algebra from a suitable Lie algebra structure on the $0$-th Hochschild homology of the algebra \cite{cb11}.
	The derived noncommutative Poisson bracket on Koszul Calabi-Yau algebras has been studied in \cite{ceey17}.
	Van den Bergh introduced double Poisson algebras in \cite{ber08}, and 
	Van de Weyer described all the double Poisson structures on finite dimensional semi-simple algebras in \cite{vdw08}.
	Recently, the notion of a noncommutative Poisson bialgebra appeared in \cite{lbs20}.

	In this paper we obtain a full description of the Poisson structures on the finitary incidence algebra $FI(P,R)$ of an arbitrary poset $P$ over a commutative unital ring $R$. This class of algebras was introduced in~\cite{KN-fi} as a generalization of incidence algebras. Recently, automorphisms~\cite{Kh-aut}, local automorphisms~\cite{CDH}, derivations~\cite{Khripchenko12}, local derivations~\cite{Khr17-loc}, Jordan derivations~\cite{Khrypchenko16}, Lie derivations~\cite{Zhang-Khrypchenko,KW} and higher derivations~\cite{KKW} of these algebras have been studied.  
	
	In \cref{prelim} we give all the basic definitions and examples which serve as a motivation and establish a connection between our work and~\cite{Benkovic09,yyz07}.
	\cref{gen-lem} of the paper is devoted to some general properties of antisymmetric biderivations defined on an associative ring $R$ with values in an $R$-module $M$. \cref{bider-FI} is the main technical part of the paper. We first describe in \cref{antisymm-bider-of-tl-I} the antisymmetric biderivations of $FI(P,R)$ defined on the subalgebra $\tilde I(P,R)$ and then prove in \cref{descr-of-bider-FI} that any antisymmetric biderivation of $FI(P,R)$ is of the same form. It turns out that any antisymmetric biderivation of $FI(P,R)$ is a Poisson structure, as proved in \cref{Bider-is-Poisson}, which is the main result of \cref{Poisson-FI} and of the present paper in general.
	
	\section{Preliminaries}\label{prelim}
	
	\subsection{Biderivations}
	Let $R$ be an associative ring and $M$ an $R$-bimodule. A biadditive map $B:R^2\to M$ is called a {\it biderivation of $R$ with values in $M$}, if it is a derivation $R\to M$ with respect to each of its two variables, i.e. for all $x,y,z\in R$
	\begin{align}
	B(xy,z)&=B(x,z)y+xB(y,z),\label{B(xy_z)=B(x_z)y+xB(y_z)}\\
	B(x,yz)&=B(x,y)z+yB(x,z).\label{B(x_yz)=B(x_y)z+yB(x_z)}
	\end{align}
	If $M$ is the regular bimodule of $R$, then we say that $B$ is a \textit{biderivation of $R$}. A biderivation $B$ is said to be \textit{antisymmetric} if $B(x,x)=0$ for all $x\in R$ (it follows that $B(x,y)=-B(y,x)$ for all $x,y\in R$). Observe that \cref{B(xy_z)=B(x_z)y+xB(y_z)} is equivalent to \cref{B(x_yz)=B(x_y)z+yB(x_z)}, whenever $B$ is antisymmetric. For any $\lb\in C(R)$ the map 
	\begin{align*}
	B(x,y)=\lb[x,y],    
	\end{align*}
	where $[x,y]$ is the commutator $xy-yx$, is an antisymmetric biderivation of $R$. Such biderivations will be called {\it inner}. 
	
	\subsection{Poisson structures}
	A \textit{Poisson algebra} over a commutative ring $R$ is an $R$-module $A$ with two bilinear operations $\cdot$ and $\{,\}$, such that $(A,\cdot)$ is an associative algebra, $(A,\{,\})$ is a Lie algebra and 
	\begin{align}\label{Leibniz-law}
	\{a,b\cdot c\}=\{a,b\}\cdot c+b\cdot\{a,c\}.
	\end{align}
	Notice that \cref{Leibniz-law} holds if and only if $\{,\}$ is a biderivation of $(A,\cdot)$.
	
	Given an associative $R$-algebra $(A,\cdot)$, a \textit{Poisson structure on $A$} is an $R$-bilinear operation $\{,\}$ such that $(A,\cdot,\{,\})$ is a Poisson algebra. Equivalently, $\{,\}$ is an $R$-bilinear antisymmetric biderivation on $(A,\cdot)$ satisfying the Jacobi identity. For any $\lb\in C(A)$ the inner biderivation $\{a,b\}=\lb[a,b]$ is clearly a Poisson structure on $A$. Following~\cite{yyz07} we call such Poisson structures \textit{standard}. There is a generalization of this notion introduced in the same paper~\cite{yyz07}. Namely, a Poisson structure on $A$ is said to be {\it piecewise standard} if $A$ decomposes into a direct sum $\bigoplus_{i=1}^m A_i$ of indecomposable Lie ideals in a way that $\{a,b\}=\lb_i[a,b]$ for all $a\in A_i$ and $b\in A$, where $\lb_i\in C(A)$.
	
	\subsection{Finitary incidence algebras}
	Let $(P,\le)$ be a partially ordered set and $R$ a commutative unital ring. For any ordered pair of $x\le y$ from $P$ we introduce a symbol $e_{xy}$ and denote by $I(P,R)$ the $R$-module of formal sums
	\begin{align}\label{formal-sum}
	f=\sum_{x\le y}f(x,y)e_{xy},
	\end{align}
	where $f(x,y)\in R$. 
	
	The sum \cref{formal-sum} is called a {\it finitary series}~\cite{KN-fi}, whenever for any pair of $x,y\in P$ with $x<y$ there exists only a finite number of $u,v\in P$, such that $x\le u<v\le y$ and $f(u,v)\ne 0$. Denote by $FI(P,R)\sst I(P,R)$ the set of finitary series. Then $FI(P,R)$ is an $R$-submodule of $I(P,R)$ which is closed under the convolution of the series:
	\begin{align}\label{conv-series}
	fg=\sum_{x\le y}\left(\sum_{x\le z\le y}f(x,z)g(z,y)\right)e_{xy},
	\end{align}
	where $f,g\in FI(P,R)$. This makes $FI(P,R)$ an associative unital $R$-algebra, called the {\it finitary incidence algebra of $P$ over $R$}. The identity element of $FI(P,R)$ is the series $\delta=\sum_{x\in P}1_R e_{xx}$. If $P=\bigsqcup_{i\in I}P_i$ is the decomposition of $P$ into its connected components, then $FI(P,R)\cong\prod_{i\in I}FI(P_i,R)$. It is an easy exercise that the center of $FI(P,R)$ is $R\delta=\{r\delta\mid r\in R\}$, provided that $P$ is connected. If $P$ is locally finite, then $FI(P,R)=I(P,R)$ is the usual incidence algebra.
	
	\subsection{Examples}
	
	It is well-known that the \textit{upper triangular matrix algebra} $T_n(R)$ is the incidence algebra $I(C_n,R)$ of a chain $C_n$ of cardinality $n$. There is a description of (not necessarily antisymmetric) biderivations of $T_n(R)$ given by Benkovi\v{c} in \cite[Corollary 4.13]{Benkovic09} (case $n\ge 3$) and \cite[Proposition 4.16]{Benkovic09} (case $n=2$). A straightforward calculation, based on this description, shows that all the antisymmetric biderivations of $T_n(R)$ are inner. Consequently, we have the following (see also \cite{kubo96}).
	
	\begin{exm}
		All the Poisson structures on $T_n(R)$ are standard.
	\end{exm}

	This is not the case for a general incidence algebra, as the next easy example shows.
	\begin{exm}\label{2-crown}
		Let $R$ be a commutative ring and $P=\{1,2,3,4\}$ with the following Hasse diagram (a crown).
		
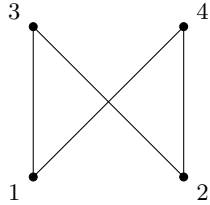
\begin{center}
			\begin{tikzpicture}[line cap=round,line join=round,>=triangle 45,x=1cm,y=1cm]
			\draw  (-1,-1)-- (1,1);
			\draw  (1,1)-- (1,-1);
			\draw  (1,-1)-- (-1,1);
			\draw  (-1,1)-- (-1,-1);
			\begin{scriptsize}
				\draw [fill=black] (-1,-1) circle (1.5pt);
				\draw[color=black] (-1.25,-1.2) node {$1$};
				\draw [fill=black] (-1,1) circle (1.5pt);
				\draw[color=black] (-1.25,1.2) node {$3$};
				\draw [fill=black] (1,1) circle (1.5pt);
				\draw[color=black] (1.25,1.2) node {$4$};
				\draw [fill=black] (1,-1) circle (1.5pt);
				\draw[color=black] (1.25,-1.2) node {$2$};
			\end{scriptsize}
			\end{tikzpicture}
		\end{center}
		Then for any $\lb,\mu,\nu,\eta\in R$ the bilinear map $B:I(P,R)\times I(P,R)\to I(P,R)$, such that
		\begin{align*}
		B(e_{11},e_{13})=B(e_{13},e_{33})=-B(e_{13},e_{11})=-B(e_{33},e_{13})&=\lb e_{13},\\
		B(e_{11},e_{14})=B(e_{14},e_{44})=-B(e_{14},e_{11})=-B(e_{44},e_{14})&=\mu e_{14},\\
		B(e_{22},e_{23})=B(e_{23},e_{33})=-B(e_{23},e_{22})=-B(e_{33},e_{23})&=\nu e_{23},\\
		B(e_{22},e_{24})=B(e_{24},e_{44})=-B(e_{24},e_{22})=-B(e_{44},e_{24})&=\eta e_{24}
		\end{align*}
		and $B(x,y)=0$ for any other pair $x,y\in\{e_{11},e_{13},e_{14},e_{22},e_{23},e_{24},e_{33},e_{44}\}$,
		defines a Poisson structure on $I(P,R)$. Observe that $B(e_{11},e_{13})=\lb[e_{11},e_{13}]$, while $B(e_{11},e_{14})=\mu[e_{11},e_{14}]$, so, whenever $\lb\ne\mu$, this Poisson structure is not standard. 
	\end{exm}

	 \begin{rem}
	     If $R$ is a field and $\lb,\mu,\nu,\eta$ are pairwise distinct, then the Poisson structure $B$ from \cref{2-crown} is not even piecewise standard.
	 \end{rem}
	
	 \begin{proof}
	     Assume that $B$ is piecewise standard. Then there are indecomposable Lie ideals $A_i$ of $I(P,R)$ and scalars $\lb_i\in R$, $1\le  i\le m$, such that $I(P,R)=\bigoplus_{i=1}^mA_i$ and $B(f,g)=\lb_i[f,g]$ for all $f\in A_i$ and $g\in I(P,R)$. Notice that one of the ideals $A_i$ contains $f$ with $f(1,3)\ne 0$ (otherwise $f(1,3)=0$ for all $f\in I(P,R)$). Let $A_1$ be such an ideal. Then $f(1,3)e_{13}=[[e_{11},f],e_{33}]\in A_1$, so that $e_{13}\in A_1$ and $\lb_1=\lb$. Similarly, one of the ideals $A_i$ contains $e_{14}$. Observe that this cannot be $A_1$, since $B(e_{14},e_{44})=\mu[e_{14},e_{44}]$ and $\lb\ne\mu$. Let $e_{14}\in A_2$, so that $\lb_2=\mu$. Applying the same argument, we may assume that $e_{23}\in A_3$ with $\lb_3=\nu$ and $e_{24}\in A_4$ with $\lb_4=\eta$. Now, if $f\in A_i$ with $i>4$, then $f(1,3)=f(1,4)=f(2,3)=f(2,4)=0$, as $e_{13},e_{14},e_{23},e_{24}\not\in A_i$. Since $(f(1,1)-f(3,3))e_{13}=[f,e_{13}]\in A_i$, then $f(1,1)=f(3,3)$. Similarly, $f(1,1)=f(4,4)$, $f(2,2)=f(3,3)$ and $f(2,2)=f(4,4)$, so $f$ belongs to the center of $I(P,R)$. It follows that $m\le 5$ and, whenever $m=5$, $A_5=R\delta=C(I(P,R))$.
	
	     Now consider $e_{11}\in I(P,R)$ and write $e_{11}=\sum_{i=1}^5 f_i$, where $f_i\in A_i$ and $f_5$ is central (possibly zero). If $f_i(1,3)\ne 0$ for some $i$, then $e_{13}\in A_i$, whence $i=1$, since the sum of $A_i$ is direct. Therefore, $0=e_{11}(1,3)=f_1(1,3)\ne 0$, a contradiction. Thus, $f_i(1,3)=0$ for all $i$. Similarly, $f_i(1,4)=f_i(2,3)=f_i(2,4)=0$ for all $i$. As above, $(f_1(1,1)-f_1(4,4))e_{14}=[f_1,e_{14}]\in A_1$, so $f_1(1,1)=f_1(4,4)$. Analogously, $f_1(2,2)=f_1(3,3)$ and $f_1(2,2)=f_1(4,4)$, and we conclude that $f_1$ is central. By the same argument $f_2$, $f_3$ and $f_4$ are central. Thus, $e_{11}$ is central, a contradiction.
	 \end{proof}
	
	\section{Some general lemmas}\label{gen-lem}
	
	Throughout this section $R$ is an associative ring, $M$ an $R$-bimodule and $B$ is an antisymmetric biderivation of $R$ with values in $M$.
	\begin{lem}\label{B(e_f)=zero}
		For any pair of orthogonal idempotents $e,f\in R$ one has $B(e,f)=0$.
	\end{lem}
	\begin{proof}
		Since $ef=0$, then 
		\begin{align*}
		0=B(ef,f)=eB(f,f)+B(e,f)f=B(e,f)f.
		\end{align*}
		Therefore, 
		\begin{align*}
		B(e,f)=B(e,f^2)=fB(e,f)+B(e,f)f=fB(e,f).
		\end{align*}
		Hence, 
		\begin{align*}
		B(e,f)=B(e^2,f)=eB(e,f)+B(e,f)e=efB(e,f)+fB(e,f)e=fB(e,f)e.
		\end{align*}
		However, by antisymmetry $B(e,f)=-B(f,e)=-eB(f,e)f$. So, $B(e,f)=-efB(f,e)fe=0$.
	\end{proof}
	
	\begin{lem}\label{B(e_fxg)=fB(e_x)g}
		For any triple of idempotents $e,f,g\in R$, such that any two of them are either equal or orthogonal, and for all $x\in R$ one has $B(e,fxg)=fB(e,x)g$. Moreover, if $e$ is orthogonal to $f$ and $g$, then $B(e,fxg)=0$.
	\end{lem}
	\begin{proof}
		Indeed, using \cref{B(e_f)=zero} we have
		\begin{align*}
		B(e,fxg)&=fB(e,xg)+B(e,f)xg=fB(e,xg)\\
		&=fxB(e,g)+fB(e,x)g=fB(e,x)g.
		\end{align*}
		Moreover, when $fe=eg=0$, we obtain
		\begin{align*}
		fB(e,x)g=fB(e^2,x)g=feB(e,x)g+fB(e,x)eg=0.
		\end{align*}
	\end{proof}
	
	\begin{lem}\label{B(e_exf)=B(exf_f)}
		For any pair idempotents $e,f\in R$ which are either equal or orthogonal and for all $x\in R$ one has $B(e,exf)=B(exf,f)$.
	\end{lem}
	\begin{proof}
		Using $ef=0$ and \cref{B(e_fxg)=fB(e_x)g}, we have
		\begin{align*}
		0=B(ef,exf)&=eB(f,exf)+B(e,exf)f=eB(f,x)f+eB(e,x)f\\
		&=B(f,exf)+B(e,exf)=-B(exf,f)+B(e,exf).
		\end{align*}
	\end{proof}
	
	\begin{lem}\label{B(exf_fyg)=eB(e_x)fyg}
		For any triple of orthogonal idempotents $e,f,g\in R$ and for all $x,y\in R$ one has $B(exf,fyg)=eB(e,x)fyg$.
	\end{lem}
	\begin{proof}
		Thanks to \cref{B(e_fxg)=fB(e_x)g,B(e_exf)=B(exf_f)} we have
		\begin{align*}
		B(exf,fyg)&=B(exf,f)yg+fB(exf,yg)=eB(e,x)fyg+fB(exf,yg).
		\end{align*}
		Now, 
		\begin{align*}
		fB(exf,yg)&=fyB(exf,g)+fB(exf,y)g=-fyB(g,exf)+fB(exf,y)g,
		\end{align*}
		the latter being $fB(exf,y)g$ by \cref{B(e_fxg)=fB(e_x)g}. We finally calculate
		\begin{align*}
		fB(exf,y)g&=feB(xf,y)g+fB(e,y)xfg=0.
		\end{align*}
	\end{proof}
	
	\begin{cor}
		For any triple of orthogonal idempotents $e,f,g\in R$ and for all $x,y\in R$ one has $B(exf,gye)=-gB(g,y)exf$.
	\end{cor}
	\begin{proof}
		Indeed, $B(exf,gye)=-B(gye,exf)=-gB(g,y)exf$.
	\end{proof}
	
	\begin{lem}\label{B(exf_gyh)=egB(exf_gyh)fh}
		Let $e,f,g,h\in R$ be idempotents, such that any two of them are either equal or orthogonal. If $e$ and $g$ are orthogonal to $f$ and $h$, then for all $x,y\in R$ one has $B(exf,gyh)=egB(exf,gyh)fh$.
	\end{lem}
	\begin{proof}
		Since $g$ is orthogonal to $f$ we have by \cref{B(e_fxg)=fB(e_x)g}
		\begin{align*}
		B(exf,gyh)&=B(exf,g)gyh+gB(exf,gyh)=-eB(g,x)fgyh+gB(exf,gyh)\\
		&=gB(exf,gyh)=gexfB(f,gyh)+gB(exf,gyh)f\\
		&=gexfgB(f,y)h+gB(exf,gyh)f=gB(exf,gyh)f.
		\end{align*}
		On the other hand, since $e$ is orthogonal to $h$, we similarly have
		\begin{align*}
		B(exf,gyh)&=-B(gyh,exf)=-eB(gyh,exf)h=eB(exf,gyh)h.
		\end{align*}
		Hence, $B(exf,gyh)=egB(exf,gyh)fh$.
		
	\end{proof}

	\section{Antisymmetric biderivations of $FI(P,R)$}\label{bider-FI}
	
	For any pair $x\le y$ we shall identify $e_{xy}$ with $1_Re_{xy}\in FI(P,R)$ and denote by $E$ the set $\{e_{xy}\mid x\le y\}\sst FI(P,R)$. Observe that
	\begin{align}\label{e_xy-times-e_uv}
	e_{xy}e_{uv}=\delta_{yu}e_{xv},
	\end{align}
	where $\delta$ is the Kronecker delta. It follows that the elements $e_x:=e_{xx}$ form a set of orthogonal idempotents of $FI(P,R)$. Moreover, for any $f\in FI(P,R)$ one has
	\begin{align}\label{e_x-alpha-e_y}
	e_xf e_y=
	\begin{cases}
	f(x,y)e_{xy}, & x\le y,\\
	0, & x\not\le y.
	\end{cases}
	\end{align}
	
%	We have $E=E_=\sqcup E_<$, where $E_==\{e_x\mid x\in P\}$ and $E_<=\{e_{xy}\mid x<y\}$. This induces the following decomposition of $FI(P,R)$ into the direct sum of $R$-modules: $FI(P,R)=D(P,R)\oplus J(P,R)$, where $D(P,R)=\{f\in FI(P,R)\mid f(x,y)=0\text{ if }x<y\}$ is a commutative subalgebra of $FI(P,R)$ and $J(P,R)=\{f\in FI(P,R)\mid f(x,y)=0\text{ if }x=y\}$ is an ideal of $FI(P,R)$ which coincides with the Jacobson radical of $FI(P,R)$ (see~\cite{KN-fi}).  
	
	Denote by $\tilde I(P,R)$ be the subalgebra of $FI(P,R)$ generated by $E$. Clearly, $\tl I(P,R)=\spn\{E\}$ as an $R$-module in view of \cref{e_xy-times-e_uv}. We will first deal with biderivations defined on $\tilde I(P,R)$. %We also introduce $\tl D(P,R)=\tl I(P,R)\cap D(P,R)$ and $\tl J(P,R)=\tl I(P,R)\cap J(P,R)$. We immediately have $\tl D(P,R)=\spn\{E_=\}$ and $\tl J(P,R)=\spn\{E_<\}$ so that $\tl I(P,R)=\tl D(P,R)\oplus\tl J(P,R)$.
	
	\subsection{Antisymmetric biderivations of $\tilde I(P,R)$ with values in $FI(P,R)$}
	Throughout this subsection we fix an $R$-bilinear antisymmetric biderivation $B$ of $\tilde I(P,R)$ with values in $FI(P,R)$. Clearly, $B$ is uniquely determined by its values on the pairs of elements of $E$.
	
	\subsubsection{The action of $B$ on the standard basis $E$}
	
	\begin{prop}\label{B=lb[]}
		For all $x\le y$ and $u\le v$ we have
		\begin{align}\label{B(e_xy_e_uv)=lb(e_xy_e_uv)[e_xy_e_uv]}
		B(e_{xy},e_{uv})=\lb(e_{xy},e_{uv})[e_{xy},e_{uv}]
		\end{align} 
		for some map $\lb:E^2\to R$.
	\end{prop}
	\begin{proof}
		\textbf{Case 1.} $x=y$ and $u=v$. Then $e_{xy}$ and $e_{uv}$ is a pair of orthogonal idempotents, so $[e_{xy},e_{uv}]=0$ and $B(e_{xy},e_{uv})=0$ by \cref{B(e_f)=zero}. Hence, $\lb(e_{xy},e_{uv})$ can be chosen arbitrarily.
		
		\textbf{Case 2.} $x=y$ and $u<v$. 
		
		\textbf{Case 2.1.} $x=u$. Then $[e_{xy},e_{uv}]=[e_x,e_{xv}]=e_{xv}$ and
		\begin{align*}
		B(e_{xy},e_{uv})=B(e_x,e_{xv})=B(e_x,e_xe_{xv}e_v)=e_xB(e_x,e_{xv})e_v=B(e_x,e_{xv})(x,v)e_{xv}
		\end{align*}
		by \cref{B(e_fxg)=fB(e_x)g,e_x-alpha-e_y}. Hence, $\lb(e_{xy},e_{uv})=B(e_x,e_{xv})(x,v)$. 
		
		\textbf{Case 2.2.} $x=v$. Then $[e_{xy},e_{uv}]=[e_x,e_{ux}]=-e_{ux}$ and
		\begin{align*}
		B(e_{xy},e_{uv})&=B(e_x,e_{ux})=-B(e_{ux},e_x)=-B(e_ue_{ux}e_x,e_x)=-B(e_u,e_ue_{ux}e_x)\\
		&=-e_uB(e_u,e_{ux})e_x=-B(e_u,e_{ux})(u,x)e_{ux}
		\end{align*}
		by \cref{B(e_fxg)=fB(e_x)g,B(e_exf)=B(exf_f),e_x-alpha-e_y}. Hence, $\lb(e_{xy},e_{uv})=B(e_u,e_{ux})(u,x)$. 
		
		\textbf{Case 2.3.} $x\not\in\{u,v\}$. Then $[e_{xy},e_{uv}]=0$ and 
		\begin{align*}
		B(e_{xy},e_{uv})=B(e_x,e_{uv})=B(e_x,e_ue_{uv}e_v)=0
		\end{align*}
		by \cref{B(e_fxg)=fB(e_x)g}. Hence, $\lb(e_{xy},e_{uv})$ can be chosen arbitrarily.
		
		\textbf{Case 3.} $x<y$ and $u=v$. We have $[e_{xy},e_{uv}]=[e_{xy},e_u]=-[e_u,e_{xy}]$ and $B(e_{xy},e_{uv})=-B(e_u,e_{xy})$, so this case reduces to Case 2.
		
		\textbf{Case 4.} $x<y$ and $u<v$.
		
		\textbf{Case 4.1.} $y=u$. Then $x<v$ and $[e_{xy},e_{uv}]=[e_{xy},e_{yv}]=e_{xv}$. We have
		\begin{align*}
		B(e_{xy},e_{uv})&=B(e_{xy},e_{yv})=B(e_xe_{xy}e_y,e_ye_{yv}e_v)=e_xB(e_x,e_{xy})e_ye_{yv}e_v\\
		&=B(e_x,e_{xy})(x,y)e_{xv}
		\end{align*}
		by \cref{B(exf_fyg)=eB(e_x)fyg}. Hence, $\lb(e_{xy},e_{uv})=B(e_x,e_{xy})(x,y)$.
		
		\textbf{Case 4.2.} $x=v$. Then $[e_{xy},e_{uv}]=[e_{xy},e_{ux}]=-[e_{ux},e_{xy}]$ and $B(e_{xy},e_{uv})=-B(e_{ux},e_{xy})$, so this case reduces to Case 4.1.
		
		\textbf{Case 4.3.} $y\ne u$ and $x\ne v$. Then $[e_{xy},e_{uv}]=0$ and
		\begin{align*}
		B(e_{xy},e_{uv})&=B(e_xe_{xy}e_y,e_ue_{uv}e_v)=e_xe_uB(e_{xy},e_{uv})e_ye_v
		\end{align*}
		by \cref{B(exf_gyh)=egB(exf_gyh)fh}. If $x\ne u$ or $y\ne v$, then $B(e_{xy},e_{uv})=0$, as $e_x$ is orthogonal to $e_u$ or $e_y$ is orthogonal to $e_v$. If $x=u$ and $y=v$, then $B(e_{xy},e_{uv})=B(e_{xy},e_{xy})=0$, since $B$ is antisymmetric.
	\end{proof}
	
	\begin{rem}\label{lb-symm}
		Let $x\le y$ and $u\le v$ such that $[e_{xy},e_{uv}]\ne 0$. Then $\lb(e_{xy},e_{uv})=\lb(e_{uv},e_{xy})$, so $\lb$ can be assumed to be symmetric.
	\end{rem}
	\begin{proof}
		Indeed, this follows from the antisymmetry of $B$: $\lb(e_{xy},e_{uv})[e_{xy},e_{uv}]=B(e_{xy},e_{uv})=-B(e_{uv},e_{xy})=-\lb(e_{uv},e_{xy})[e_{uv},e_{xy}]=\lb(e_{uv},e_{xy})[e_{xy},e_{uv}]$.
	\end{proof}
	
	\begin{lem}\label{properties-of-lb}
		Let $\lb:E^2\to R$ be the map from \cref{B=lb[]}. Then
		\begin{enumerate}
			\item for all $x<y$ one has $\lb(e_x,e_{xy})=\lb(e_{xy},e_y)$;\label{lb(e_x_e_xy)=lb(e_xy_e_y)}
			\item for all $x<y<z$ one has $\lb(e_{xy},e_{yz})=\lb(e_x,e_{xy})$;\label{lb(e_xy_e_yz)=lb(e_x_e_xy)}
			\item for all $x\le y<z<u$ one has $\lb(e_{xy},e_{yz})=\lb(e_{xy},e_{yu})$;\label{lb(e_xy_e_yz)=lb(e_xy_e_yu)}
			\item for all $x<y<z\le u$ one has $\lb(e_{yz},e_{zu})=\lb(e_{xz},e_{zu})$.\label{lb(e_yz_e_zu)=lb(e_xz_e_zu)}
		\end{enumerate}
	\end{lem}
	\begin{proof}
		Item \cref{lb(e_x_e_xy)=lb(e_xy_e_y)} follows from \cref{B(e_exf)=B(exf_f)}, since $B(e_x,e_{xy})=B(e_x,e_xe_{xy}e_y)$ and $B(e_{xy},e_y)=B(e_xe_{xy}e_y,e_y)$.
		
		Item \cref{lb(e_xy_e_yz)=lb(e_x_e_xy)} follows from \cref{B(exf_fyg)=eB(e_x)fyg} (see Cases 2.1 and 4.1 of \cref{B=lb[]}).
		
		Item \cref{lb(e_xy_e_yz)=lb(e_xy_e_yu)} is proved by observing that $B(e_{xy},e_{zu})=0$ as $[e_{xy},e_{zu}]=0$, so
		\begin{align*}
		\lb(e_{xy},e_{yu})e_{xu}&=\lb(e_{xy},e_{yu})[e_{xy},e_{yu}]=	B(e_{xy},e_{yu})=B(e_{xy},e_{yz}e_{zu})\\
		&=e_{yz}B(e_{xy},e_{zu})+B(e_{xy},e_{yz})e_{zu}=\lb(e_{xy},e_{yz})[e_{xy},e_{yz}]e_{zu}\\
		&=\lb(e_{xy},e_{yz})e_{xz}e_{zu}=\lb(e_{xy},e_{yz})e_{xu}.
		\end{align*}
		
		Item \cref{lb(e_yz_e_zu)=lb(e_xz_e_zu)} is proved similarly to item \cref{lb(e_xy_e_yz)=lb(e_xy_e_yu)} by considering $B(e_{xz},e_{zu})=B(e_{xy}e_{yz},e_{zu})$.
	\end{proof}
	
	\begin{lem}\label{lb-constant-on-chains}
		Let $\lb:E^2\to R$ be a symmetric map. Then conditions \cref{lb(e_x_e_xy)=lb(e_xy_e_y),lb(e_xy_e_yz)=lb(e_x_e_xy),lb(e_xy_e_yz)=lb(e_xy_e_yu),lb(e_yz_e_zu)=lb(e_xz_e_zu)} of \cref{properties-of-lb} are equivalent to the following property: for any chain $C\sst P$ and for all $x\le y$, $x'\le y'$, $u\le v$, $u'\le v'$ from $C$
		\begin{align}\label{lb(e_xy_e_uv)=lb(e_uv_e_u'v')}
		[e_{xy},e_{x'y'}]\ne 0\ \&\ [e_{uv},e_{u'v'}]\ne 0\ \impl\ \lb(e_{xy},e_{x'y'})=\lb(e_{uv},e_{u'v'}).
		\end{align}
	\end{lem}
	\begin{proof}
		Clearly \cref{lb(e_xy_e_uv)=lb(e_uv_e_u'v')} implies \cref{lb(e_x_e_xy)=lb(e_xy_e_y),lb(e_xy_e_yz)=lb(e_x_e_xy),lb(e_xy_e_yz)=lb(e_xy_e_yu),lb(e_yz_e_zu)=lb(e_xz_e_zu)}.
		
		Conversely, let $x\le y$ and $x'\le y'$ such that $[e_{xy},e_{x'y'}]\ne 0$. If $x<y=x'$, then $\lb(e_{xy},e_{x'y'})=\lb(e_x,e_{xy})$ by \cref{properties-of-lb}\cref{lb(e_xy_e_yz)=lb(e_x_e_xy)}. If $x'<y'=x$, then by the symmetry of $\lb$ and the result of the previous case $\lb(e_{xy},e_{x'y'})=\lb(e_{x'y'},e_{xy})=\lb(e_{x'},e_{x'y'})$. If $x=y=x'<y'$, then we immediately obtain $\lb(e_{xy},e_{x'y'})=\lb(e_{x'},e_{x'y'})$. Finally, if $x<y=x'=y'$, then $\lb(e_{xy},e_{x'y'})=\lb(e_{x'y'},e_{y'})=\lb(e_{x'},e_{x'y'})$ thanks to \cref{properties-of-lb}\cref{lb(e_x_e_xy)=lb(e_xy_e_y)}. Thus, the proof of \cref{lb(e_xy_e_uv)=lb(e_uv_e_u'v')} reduces to the following: for any chain $C\sst P$ and for all $x<y$, $u<v$ from $C$
		\begin{align}\label{lb(e_x_e_xy)=lb(e_u_e_uv)}
		\lb(e_x,e_{xy})=\lb(e_u,e_{uv}).
		\end{align}
		Let $z=\min\{x,u\}$ and $w=\max\{y,v\}$. Notice that such elements exist because $x,u\in C$ and $y,v\in C$. Then thanks to \cref{lb(e_xy_e_yz)=lb(e_xy_e_yu),lb(e_x_e_xy)=lb(e_xy_e_y),lb(e_yz_e_zu)=lb(e_xz_e_zu)}
		\begin{align*}
		\lb(e_x,e_{xy})=\lb(e_x,e_{xw})=\lb(e_{xw},e_w)=\lb(e_{zw},e_w).
		\end{align*}
		Analogously, $\lb(e_u,e_{uv})=\lb(e_{zw},e_w)$, whence \cref{lb(e_x_e_xy)=lb(e_u_e_uv)}.
	\end{proof}
	
	\begin{prop}\label{lb-const-on-chains}
		Let $\lb:E^2\to R$ be a symmetric map. Then the bilinear map $B$ given by \cref{B(e_xy_e_uv)=lb(e_xy_e_uv)[e_xy_e_uv]} is an antisymmetric biderivation of $\tilde I(P,R)$ with values in $FI(P,R)$ if and only if \cref{lb(e_xy_e_uv)=lb(e_uv_e_u'v')} holds for any chain $C\subseteq P$.
	\end{prop}
	\begin{proof}
		The ``only if'' part has already been proved in \cref{properties-of-lb,lb-constant-on-chains}.
		
		Conversely, let $B$ be a bilinear map $\tilde I(P,R)^2\to FI(P,R)$ given on the pairs from $E$ by formula \cref{B(e_xy_e_uv)=lb(e_xy_e_uv)[e_xy_e_uv]}, where $\lb$ is symmetric and satisfies \cref{lb(e_xy_e_uv)=lb(e_uv_e_u'v')}. It follows from the symmetry of $\lb$ that $B$ is antisymmetric. Let $x\le y$, $z\le w$ and $u\le v$. We are going to prove
		\begin{align}\label{B(e_xye_zw_e_uv)=e_xyB(e_zw_e_uv)+B(e_xy_e_uv)e_zw}
		B(e_{xy}e_{zw},e_{uv})=e_{xy}B(e_{zw},e_{uv})+B(e_{xy},e_{uv})e_{zw}.
		\end{align}
		
		\textbf{Case 1.} $y=z$. Then $e_{xy}e_{zw}=e_{xw}$, so $B(e_{xy}e_{zw},e_{uv})=B(e_{xw},e_{uv})$.
		
		\textbf{Case 1.1.} $x=w=u=v$. Then both sides of \cref{B(e_xye_zw_e_uv)=e_xyB(e_zw_e_uv)+B(e_xy_e_uv)e_zw} are zero, as $[e_{xy}e_{zw},e_{uv}]=[e_{zw},e_{uv}]=[e_{xy},e_{uv}]=0$. 
		
		\textbf{Case 1.2.} $x=w=u<v$. Then
		\begin{align*}
		e_{xy}B(e_{zw},e_{uv})+B(e_{xy},e_{uv})e_{zw}&=e_x\lb(e_x,e_{xv})[e_x,e_{xv}]+\lb(e_x,e_{xv})[e_x,e_{xv}]e_x\\
		&=\lb(e_x,e_{xv})(e_xe_{xv}+e_{xv}e_x)=\lb(e_x,e_{xv})e_{xv}\\
		&=\lb(e_x,e_{xv})[e_x,e_{xv}]=B(e_{xw},e_{uv}).
		\end{align*} 
		
		\textbf{Case 1.3.} $u<v=x=w$. Then
		\begin{align*}
		e_{xy}B(e_{zw},e_{uv})+B(e_{xy},e_{uv})e_{zw}&=e_x\lb(e_x,e_{ux})[e_x,e_{ux}]+\lb(e_x,e_{ux})[e_x,e_{ux}]e_x\\
		&=\lb(e_x,e_{ux})(-e_xe_{ux}-e_{ux}e_x)=-\lb(e_x,e_{ux})e_{ux}\\
		&=\lb(e_x,e_{ux})[e_x,e_{ux}]=B(e_{xw},e_{uv}).
		\end{align*}
		
		\textbf{Case 1.4.} $u<v$ and $x=w\not\in\{u,v\}$. Then $[e_{xy}e_{zw},e_{uv}]=[e_{zw},e_{uv}]=[e_{xy},e_{uv}]=0$.
		
		\textbf{Case 1.5.} $x<w=u=v$. Then
		\begin{align*}
		e_{xy}B(e_{zw},e_{uv})+B(e_{xy},e_{uv})e_{zw}&=e_{xy}B(e_{yu},e_u)+B(e_{xy},e_u)e_{yu}.
		\end{align*}
		If $y=u$, then
		\begin{align*}
		e_{xy}B(e_{yu},e_u)+B(e_{xy},e_u)e_{yu}&=e_{xy}B(e_y,e_y)+B(e_{xy},e_y)e_y\\
		&=\lb(e_{xy},e_y)[e_{xy},e_y]e_y=\lb(e_{xy},e_y)e_{xy}\\
		&=\lb(e_{xy},e_y)[e_{xy},e_y]=B(e_{xw},e_{uv}).
		\end{align*} 
		Otherwise, 
		\begin{align*}
		e_{xy}B(e_{yu},e_u)+B(e_{xy},e_u)e_{yu}&=e_{xy}\lb(e_{yu},e_u)[e_{yu},e_u]+\lb(e_{xy},e_u)[e_{xy},e_u]e_{zu}\\
		&=\lb(e_{yu},e_u)e_{xy}e_{yu}=\lb(e_{yu},e_u)e_{xu},
		\end{align*} 
		where the latter equals $\lb(e_{xu},e_u)e_{xu}=\lb(e_{xu},e_u)[e_{xu},e_u]=B(e_{xw},e_{uv})$ thanks to \cref{lb(e_xy_e_uv)=lb(e_uv_e_u'v')}.
		
		\textbf{Case 1.6.} $u=v=x<w$. Then
		\begin{align*}
		e_{xy}B(e_{zw},e_{uv})+B(e_{xy},e_{uv})e_{zw}&=e_{xy}B(e_{yw},e_x)+B(e_{xy},e_x)e_{yw}.
		\end{align*}
		If $y=x$, then
		\begin{align*}
		e_{xy}B(e_{yw},e_x)+B(e_{xy},e_x)e_{yw}&=e_xB(e_{xw},e_x)+B(e_x,e_x)e_{xw}\\
		&=\lb(e_{xw},e_x)e_x[e_{xw},e_x]=-\lb(e_{xw},e_x)e_{xw}\\
		&=\lb(e_{xw},e_x)[e_{xw},e_x]=B(e_{xw},e_{uv}).
		\end{align*}
		Otherwise, using the symmetry of $\lb$ and \cref{lb(e_xy_e_uv)=lb(e_uv_e_u'v')} we have
		\begin{align*}
		e_{xy}B(e_{yw},e_x)+B(e_{xy},e_x)e_{yw}&=e_{xy}\lb(e_{yw},e_x)[e_{yw},e_x]+\lb(e_{xy},e_x)[e_{xy},e_x]e_{yw}\\
		&=-\lb(e_{xy},e_x)e_{xy}e_{yw}=-\lb(e_{xy},e_x)e_{xw}\\
		&=-\lb(e_x,e_{xy})e_{xw}=-\lb(e_x,e_{xw})e_{xw}\\
		&=-\lb(e_{xw},e_x)e_{xw}=\lb(e_{xw},e_x)[e_{xw},e_x]\\
		&=B(e_{xw},e_{uv}).
		\end{align*}
		
		\textbf{Case 1.7.} $x<w$ and $u=v\not\in\{x,w\}$. Then $B(e_{xw},e_{uv})=0$, as $[e_{xw},e_{uv}]=0$. Furthermore,
		\begin{align*}
		e_{xy}B(e_{zw},e_{uv})+B(e_{xy},e_{uv})e_{zw}&=e_{xy}B(e_{yw},e_u)+B(e_{xy},e_u)e_{yw}.
		\end{align*} 
		If $y=u$, then using he symmetry of $\lb$ and \cref{lb(e_xy_e_uv)=lb(e_uv_e_u'v')} we have
		\begin{align*}
		e_{xy}B(e_{yw},e_u)+B(e_{xy},e_u)e_{yw}&=e_{xy}B(e_{yw},e_y)+B(e_{xy},e_y)e_{yw}\\
		&=e_{xy}\lb(e_{yw},e_y)[e_{yw},e_y]+\lb(e_{xy},e_y)[e_{xy},e_y]e_{yw}\\
		&=-\lb(e_{yw},e_y)e_{xy}e_{yw}+\lb(e_{xy},e_y)e_{xy}e_{yw}\\
		&=(\lb(e_{xy},e_y)-\lb(e_{yw},e_y))e_{xw}\\
		&=(\lb(e_{xy},e_y)-\lb(e_y,e_{yw}))e_{xw}=0.
		\end{align*}
		Otherwise, 
		\begin{align*}
		e_{xy}B(e_{yw},e_u)+B(e_{xy},e_u)e_{yw}=e_{xy}\lb(e_{yw},e_u)[e_{yw},e_u]+\lb(e_{xy},e_u)[e_{xy},e_u]e_{yw}=0.
		\end{align*}
		
		\textbf{Case 1.8.} $u<v=x<w$. Then
		\begin{align*}
		e_{xy}B(e_{zw},e_{uv})+B(e_{xy},e_{uv})e_{zw}&=e_{xy}B(e_{yw},e_{ux})+B(e_{xy},e_{ux})e_{yw}\\
		&=\lb(e_{yw},e_{ux})e_{xy}[e_{yw},e_{ux}]+\lb(e_{xy},e_{ux})[e_{xy},e_{ux}]e_{yw}\\
		&=\lb(e_{yw},e_{ux})e_{xy}[e_{yw},e_{ux}]-\lb(e_{xy},e_{ux})e_{uw}.
		\end{align*}
		Since $e_{xy}[e_{yw},e_{ux}]=-\delta_{xy}e_{xy}e_{uw}=0$, then using the symmetry of $\lb$ and \cref{lb(e_xy_e_uv)=lb(e_uv_e_u'v')} we have
		\begin{align*}
		\lb(e_{yw},e_{ux})e_{xy}[e_{yw},e_{ux}]-\lb(e_{xy},e_{ux})e_{uw}&=-\lb(e_{xy},e_{ux})e_{uw}=-\lb(e_{ux},e_{xy})e_{uw}\\
		&=-\lb(e_{ux},e_{xw})e_{uw}=-\lb(e_{xw},e_{ux})e_{uw}\\
		&=\lb(e_{xw},e_{ux})[e_{xw},e_{ux}]=B(e_{xw},e_{uv}).
		\end{align*}
		
		\textbf{Case 1.9.} $x<w=u<v$. Then
		\begin{align*}
		e_{xy}B(e_{zw},e_{uv})+B(e_{xy},e_{uv})e_{zw}&=e_{xy}B(e_{yu},e_{uv})+B(e_{xy},e_{uv})e_{yu}\\
		&=\lb(e_{yu},e_{uv})e_{xy}[e_{yu},e_{uv}]+\lb(e_{xy},e_{uv})[e_{xy},e_{uv}]e_{yu}\\
		&=\lb(e_{yu},e_{uv})e_{xv}+\lb(e_{xy},e_{uv})[e_{xy},e_{uv}]e_{yu}.
		\end{align*}
		Since $[e_{xy},e_{uv}]e_{yu}=\delta_{yu}e_{xv}e_{yu}=0$, then in view of \cref{lb(e_xy_e_uv)=lb(e_uv_e_u'v')} we obtain
		\begin{align*}
		\lb(e_{yu},e_{uv})e_{xv}+\lb(e_{xy},e_{uv})[e_{xy},e_{uv}]e_{yu}&=\lb(e_{yu},e_{uv})e_{xv}=\lb(e_{xy},e_{yv})e_{xv}\\
		&=\lb(e_{xy},e_{yv})[e_{xy},e_{yv}]=B(e_{xw},e_{uv}).
		\end{align*}
		
		\textbf{Case 1.10.} $x=u<v=w$. Then
		\begin{align*}
		e_{xy}B(e_{zw},e_{uv})+B(e_{xy},e_{uv})e_{zw}&=e_{xy}B(e_{yv},e_{xv})+B(e_{xy},e_{xv})e_{yv}\\
		&=\lb(e_{yv},e_{xv})e_{xy}[e_{yv},e_{xv}]+\lb(e_{xy},e_{xv})[e_{xy},e_{xv}]e_{yv}\\
		&=-\lb(e_{yv},e_{xv})\delta_{yv}e_{xy}e_{xv}+\lb(e_{xy},e_{xv})\delta_{xy}e_{xv}e_{yv}\\
		&=0=B(e_{xw},e_{uv}).
		\end{align*}
		
		\textbf{Case 1.11.} $x<w$, $u<v$ and $\{x,w\}\cap\{u,v\}=\emptyset$. Then $B(e_{xw},e_{uv})=0$, as $[e_{xw},e_{uv}]=0$. Furthermore,
		\begin{align*}
		e_{xy}B(e_{zw},e_{uv})+B(e_{xy},e_{uv})e_{zw}&=e_{xy}B(e_{yw},e_{uv})+B(e_{xy},e_{uv})e_{yw}\\
		&=\lb(e_{yw},e_{uv})e_{xy}[e_{yw},e_{uv}]+\lb(e_{xy},e_{uv})[e_{xy},e_{uv}]e_{yw}\\
		&=-\lb(e_{yw},e_{uv})\delta_{yv}e_{xy}e_{uw}+\lb(e_{xy},e_{uv})\delta_{yu}e_{xv}e_{yw}=0.
		\end{align*}

		\textbf{Case 2.} $y\ne z$. Then $e_{xy}e_{zw}=0$, so $B(e_{xy}e_{zw},e_{uv})=0$. 
		
		\textbf{Case 2.1.} $w=u$ and $x=v$. Then $e_{xy}[e_{zw},e_{uv}]=[e_{xy},e_{uv}]e_{zw}=0$, so $e_{xy}B(e_{zw},e_{uv})=B(e_{xy},e_{uv})e_{zw}=0$.
		
		\textbf{Case 2.2.} $w=u$, $x\ne v$ and $y\ne u$. Then $e_{xy}[e_{zw},e_{uv}]=[e_{xy},e_{uv}]e_{zw}=0$.
		
		\textbf{Case 2.3.} $w=u$, $x\ne v$ and $y=u$. Then $e_{xy}[e_{zw},e_{uv}]=0$ and $[e_{xy},e_{uv}]e_{zw}=e_{xv}e_{zw}$. If $v=z$, then it follows from $z\le w=u\le v$ that $z=u=y$, a contradiction. So, $v\ne z$, whence $e_{xv}e_{zw}=0$.
		
		\textbf{Case 2.4.} $w\ne u$, $x=v$ and $z\ne v$. Then $[e_{xy},e_{uv}]e_{zw}=e_{xy}[e_{zw},e_{uv}]=0$.
		
		\textbf{Case 2.5.} $w\ne u$, $x=v$ and $z=v$. Then $[e_{xy},e_{uv}]e_{zw}=0$ and $e_{xy}[e_{zw},e_{uv}]=-e_{xy}e_{uw}$. If $y=u$, then it follows from $u\le v=x\le y$ that $y=v=z$, a contradiction. So $y\ne u$, whence $e_{xy}e_{uw}=0$.
		
		\textbf{Case 2.6.} $w\ne u$, $x\ne v$, $y\ne u$ and $z\ne v$. Then $[e_{zw},e_{uv}]=[e_{xy},e_{uv}]=0$.
		
		\textbf{Case 2.7.} $w\ne u$, $x\ne v$, $y\ne u$ and $z=v$. Then $[e_{xy},e_{uv}]=0$ and $e_{xy}[e_{zw},e_{uv}]=-e_{xy}e_{uw}=0$.
		
		\textbf{Case 2.8.} $w\ne u$, $x\ne v$, $y=u$ and $z\ne v$. Then $[e_{zw},e_{uv}]=0$ and $[e_{xy},e_{uv}]e_{zw}=e_{xv}e_{zw}=0$.
		
		\textbf{Case 2.9.} $w\ne u$, $x\ne v$, $y=u$ and $z=v$. Then using the symmetry of $\lb$ and \cref{lb(e_xy_e_uv)=lb(e_uv_e_u'v')} we have
		\begin{align*}
		e_{xy}B(e_{zw},e_{uv})+B(e_{xy},e_{uv})e_{zw}&=\lb(e_{zw},e_{yz})e_{xy}[e_{zw},e_{yz}]+\lb(e_{xy},e_{yz})[e_{xy},e_{yz}]e_{zw}\\
		&=(\lb(e_{xy},e_{yz})-\lb(e_{zw},e_{yz}))e_{xw}\\
		&=(\lb(e_{xy},e_{yz})-\lb(e_{yz},e_{zw}))e_{xw}=0.
		\end{align*} 
		
	\end{proof}
	
	%\begin{defn}
	%	A symmetric map $\lb:E^2\to R$ satisfying \cref{lb(e_xy_e_uv)=lb(e_uv_e_u'v')} will be said to be \textit{constant on chains}.
	%\end{defn}
	
	\subsubsection{The action of $B$ on an arbitrary pair of elements}
	
	Denote by $P^2_<$ the set $\{(x,y)\in P^2\mid x<y\}$.
	\begin{defn}
		A map $\sg:P^2_<\to R$ is said to be \textit{constant on chains} if for any chain $C\sst P$ and for all $x<y$, $u<v$ from $C$ one has $\sg(x,y)=\sg(u,v)$.
	\end{defn}

	\begin{thrm}\label{antisymm-bider-of-tl-I}
		There is a one-to-one correspondence between the antisymmetric biderivations $B$ of $\tilde I(P,R)$ with values in $FI(P,R)$ and the maps $\sg:P^2_<\to R$ which are constant on chains. More precisely, 
		\begin{align}\label{B(f_g)(x_y)=sg(x_y)[f_g](x_y)}
		B(f,g)(x,y)=
		\begin{cases}
		0, & x=y,\\
		\sg(x,y)[f,g](x,y), & x<y,
		\end{cases}
		\end{align}
		where $f,g\in \tilde I(P,R)$.
	\end{thrm}
	\begin{proof}
		Let $B$ be an antisymmetric biderivation of $\tilde I(P,R)$ with values in $FI(P,R)$. Define
		\begin{align*}
		\sg(x,y)=\lb(e_x,e_{xy}).
		\end{align*}
		Then for any chain $C\sst P$ and for all $x<y$, $u<v$ from $C$ we have $\sg(x,y)=\sg(u,v)$ by \cref{lb(e_x_e_xy)=lb(e_u_e_uv)},
		so $\sg$ is constant on chains. Now, given $f,g\in \tilde I(P,R)$, we write $f=\sum_{x\le y}f(x,y)e_{xy}$ and $g=\sum_{x\le y}g(x,y)e_{xy}$, so that by \cref{B=lb[]}
		\begin{align*}
		B(f,g)&=\sum_{x\le y}\sum_{u\le v}f(x,y)g(u,v)B(e_{xy},e_{uv})\\
		&=\sum_{x\le y}\sum_{u\le v}f(x,y)g(u,v)\lb(e_{xy},e_{uv})[e_{xy},e_{uv}],
		\end{align*}
		where we assume that $\lb$ is symmetric in view of \cref{lb-symm}.
		
		Since $[e_{xy},e_{uv}]\ne 0$ implies $y=u$ or $x=v$, and $[e_{xy},e_{uv}]=0$ whenever $y=u$ and $x=v$, we conclude that
		\begin{align*}
		B(f,g)&=\sum_{x\le y\le v}f(x,y)g(y,v)\lb(e_{xy},e_{yv})e_{xv}-\sum_{u\le x\le y}f(x,y)g(u,x)\lb(e_{xy},e_{ux})e_{uy}\\
		&=\sum_{x\le z\le y}\left(f(x,z)g(z,y)\lb(e_{xz},e_{zy})-g(x,z)f(z,y)\lb(e_{zy},e_{xz})\right)e_{xy}\\
		&=\sum_{x\le z\le y}\lb(e_{xz},e_{zy})\left(f(x,z)g(z,y)-g(x,z)f(z,y)\right)e_{xy}.
		\end{align*}
		If $x=y$, then we get $B(f,g)(x,y)=0$. If $x<y$, then considering the chain $x<z<y$ and using \cref{lb(e_xy_e_uv)=lb(e_uv_e_u'v')} we see that $\lb(e_{xz},e_{zy})=\lb(e_x,e_{xy})=\sg(x,y)$, whence
		\begin{align*}
		B(f,g)(x,y)&=\sg(x,y)\sum_{x\le z\le y}\left(f(x,z)g(z,y)-g(x,z)f(z,y)\right)\\
		&=\sg(x,y)((fg)(x,y)-(gf)(x,y))=\sg(x,y)[f,g](x,y).
		\end{align*}
		
		Conversely, assume that $B$ be given by \cref{B(f_g)(x_y)=sg(x_y)[f_g](x_y)}. It is immediately seen that $B$ is an antisymmetric bilinear map. Taking $x<y$, we define $\lb(e_x,e_{xy})=\lb(e_{xy},e_x):=\sg(x,y)$ and $\lb(e_{xz},e_{zy})=\lb(e_{zy},e_{xz}):=\lb(e_x,e_{xy})$ for all $x\le z\le y$. 
		Then $\lb$ is defined on all pairs $(e_{xy},e_{uv})$ with $[e_{xy},e_{uv}]\ne 0$, and it is symmetric on all such pairs. Moreover, 
		\begin{align*}
		B(e_x,e_{xy})(u,v)=\sg(u,v)[e_x,e_{xy}](u,v)=\sg(u,v)e_{xy}(u,v)=
		\begin{cases}
		\sg(x,y), & (x,y)=(u,v),\\
		0,        & (x,y)\ne(u,v),
		\end{cases}
		\end{align*}
		whence $B(e_x,e_{xy})=\sg(x,y)e_{xy}=\sg(x,y)[e_x,e_{xy}]=\lb(e_x,e_{xy})[e_x,e_{xy}]$.
		Similarly, $B(e_{xy},e_y)=\sg(x,y)e_{xy}=\lb(e_{xy},e_y)[e_{xy},e_y]$ and $B(e_{xz},e_{zy})=\sg(x,y)e_{xy}=\lb(e_{xz},e_{zy})[e_{xz},e_{zy}]$, so that \cref{B(e_xy_e_uv)=lb(e_xy_e_uv)[e_xy_e_uv]} holds. Finally, \cref{lb(e_xy_e_uv)=lb(e_uv_e_u'v')} follows from the definition of $\lb$ and the fact that $\sg$ is constant on chains.
	\end{proof}
	
	\subsection{Antisymmetric biderivations of $FI(P,R)$}
	
	\begin{defn}\label{f|_x^y-defn}
		Let $f\in FI(P,R)$ and $x\le y$. We define as in \cite{KW}
		\begin{align}\label{f|_x^y}
		f|_x^y=f(x,y)e_{xy}+\sum_{x\le v<y}f(x,v)e_{xv}+\sum_{x<u\le y}f(u,y)e_{uy}.
		\end{align}
		Obviously, $f|_x^y\in \tilde I(P,R)$.
	\end{defn}
	
	\begin{lem}\label{B(f_g)(x_y)=B(restr-of-f-and-g)(x_y)}
		Let $B$ be a biderivation of $FI(P,R)$. Then for all $f,g\in FI(P,R)$ and $x\le y$
		\begin{align}\label{B(f_g)(x_y)=B(f|_x^y_g|_x^y)(x_y)}
		B(f,g)(x,y)=B(f|_x^y,g|_x^y)(x,y).
		\end{align}
	\end{lem}
	\begin{proof}
		Since $B(-,g)$ is a derivation of $FI(P,R)$ we have $B(f,g)(x,y)=B(f|_x^y,g)(x,y)$ by \cite[Lemma 2.4]{KW} (case $x<y$) and \cite[Lemma 8]{Khripchenko12} (case $x=y$). Now, since $B(f|_x^y,-)$ is a derivation of $FI(P,R)$, we similarly obtain $B(f|_x^y,g)(x,y)=B(f|_x^y,g|_x^y)(x,y)$.
	\end{proof}

	\begin{thrm}\label{descr-of-bider-FI}
		Antisymmetric biderivations of $FI(P,R)$ are exactly the maps $B$ of the form \cref{B(f_g)(x_y)=sg(x_y)[f_g](x_y)}, in which $f,g\in FI(P,R)$.
	\end{thrm}
	\begin{proof}
		Let $B$ be an antisymmetric biderivations of $FI(P,R)$. Then
		\begin{center}$B(f,g)(x,y)=B(f|_x^y,g|_x^y)(x,y)$
		\end{center}by \cref{B(f_g)(x_y)=B(restr-of-f-and-g)(x_y)}. If $x=y$, then $B(f|_x^y,g|_x^y)(x,y)=0$ by \cref{antisymm-bider-of-tl-I}. If $x<y$, then $B(f|_x^y,g|_x^y)(x,y)=\sg(x,y)[f|_x^y,g|_x^y](x,y)$ by the same \cref{antisymm-bider-of-tl-I}. However, $[f|_x^y,g|_x^y](x,y)=[f,g](x,y)$ thanks to \cite[Lemma 2.3 (ii)]{KW}.
		
		Conversely, let $B$ be the map $FI(P,R)\times FI(P,R)\to I(P,R)$ defined by \cref{B(f_g)(x_y)=sg(x_y)[f_g](x_y)}, where $f,g\in FI(P,R)$ and $\sg:P^2_<\to R$ is constant on chains. Clearly, $B$ is a bilinear antisymmetric map. Now, using the facts that the commutator is a biderivation and $\sg$ is constant on chains, for all $x<y$ we have
		\begin{align*}
		B(fg,h)(x,y)&=\sg(x,y)[fg,h](x,y)=\sg(x,y)(f[g,h]+[f,h]g)(x,y)\\
		&=\sg(x,y)\left(\sum_{x\le z\le y}f(x,z)[g,h](z,y)+\sum_{x\le z\le y}[f,h](x,z)g(z,y)\right)\\
		&=\sg(x,y)\left(\sum_{x\le z<y}f(x,z)[g,h](z,y)+\sum_{x<z\le y}[f,h](x,z)g(z,y)\right)\\
		&=\sum_{x\le z<y}f(x,z)\sg(z,y)[g,h](z,y)+\sum_{x<z\le y}\sg(x,z)[f,h](x,z)g(z,y)\\
		&=\sum_{x\le z\le y}f(x,z)B(g,h)(z,y)+\sum_{x\le z\le y}B(f,h)(x,z)g(z,y)\\
		&=(fB(g,h)+B(f,h)g)(x,y).
		\end{align*}
		Obviously, 
		\begin{align*}
		(fB(g,h)+B(f,h)g)(x,x)=f(x,x)B(g,h)(x,x)+B(f,h)(x,x)g(x,x)=0,
		\end{align*}
		so $B(fg,h)(x,y)=(fB(g,h)+B(f,h)g)(x,y)$ for all $x\le y$, and thus $B$ is a biderivation of $FI(P,R)$ with values in $I(P,R)$. It only remains to show that $B(f,g)\in FI(P,R)$. Suppose that there are $x<y$ and $x\le u_i<v_i\le y$, $i\in I$, such that $B(f,g)(u_i,v_i)\ne 0$, where $I$ is infinite. Then $\sg(u_i,v_i)[f,g](u_i,v_i)\ne 0$ implies $[f,g](u_i,v_i)\ne 0$. But this contradicts the fact that $[f,g]\in FI(P,R)$.
	\end{proof}
	
	\begin{cor}
		Let $P$ be a poset, in which the intersection of any two maximal chains has cardinality at least $2$. Then each antisymmetric biderivation of $FI(P,R)$ is inner.
	\end{cor}
	\begin{proof}
		For all $x<y$ and $u<v$ there are maximal chains $C,D\sst P$, such that $x,y\in C$ and $u,v\in D$. Let $z<w$ be such that $z,w\in C\cap D$. If $\sg:P^2_<\to R$ is constant on $C$ and $D$, we conclude that $\sg(x,y)=\sg(z,w)=\sg(u,v)$.
	\end{proof}
	
%	As a consequence, we obtain the following particular case of a result by Benkovi\v{c}~\cite{Benkovic09}.
%	\begin{cor}
%		Each antisymmetric biderivation $B$ of $T_n(R)$ is of the form $B(f,g)=\sg[f,g]$ for some $\sg\in R$.
%	\end{cor}
%	\begin{proof}
%		For, $T_n(R)$ is isomorphic to $FI(C_n,R)=I(C_n,R)$, where $C_n$ is a chain of $n$ elements.
%	\end{proof}

	\section{Poisson structures on $FI(P,R)$}\label{Poisson-FI}
	
	The next proposition shows that each antisymmetric biderivation of $FI(P,R)$ with values in $FI(P,R)$ satisfies the Jacobi identity.
	
	\begin{prop}\label{B-satisfies-Jacobi}
		Let $B$ be an antisymmetric biderivation of $FI(P,R)$ with values in $FI(P,R)$. Then for all $f,g,h\in FI(P,R)$ one has
		\begin{align}\label{Jacobi-for-B-general}
		B(f,B(g,h))+B(g,B(h,f))+B(h,B(f,g))=0.
		\end{align}
	\end{prop}
	\begin{proof}
		It immediately follows from \cref{descr-of-bider-FI} that
		\begin{align*}
		B(f,B(g,h))(x,x)=B(g,B(h,f))(x,x)=B(h,B(f,g))(x,x)=0.
		\end{align*}
		Let $x<y$. Since $\sg$ is constant on chains, in view of \cref{descr-of-bider-FI} we have
		\begin{align*}
		B(f,B(g,h))(x,y)&=\sg(x,y)[f,B(g,h)](x,y)\\
		&=\sg(x,y)\left(\sum_{x\le z\le y}f(x,z)B(g,h)(z,y)-\sum_{x\le z\le y}B(g,h)(x,z)f(z,y)\right)\\
		&=\sg(x,y)\left(\sum_{x\le z<y}f(x,z)\sg(z,y)[g,h](z,y)-\sum_{x<z\le y}\sg(x,z)[g,h](x,z)f(z,y)\right)\\
		&=\sg(x,y)^2\left(\sum_{x\le z\le y}f(x,z)[g,h](z,y)-\sum_{x\le z\le y}[g,h](x,z)f(z,y)\right)\\
		&=\sg(x,y)^2[f,[g,h]](x,y).
		\end{align*}
		Similarly, $B(g,B(h,f))(x,y)=\sg(x,y)^2[g,[h,f]](x,y)$ and
		\begin{center}$B(h,B(f,g))(x,y)=\sg(x,y)^2[h,[f,g]](x,y)$.
		\end{center}
		Since the commutator satisfies the Jacobi identity, we conclude that 
		\begin{align*}
		&(B(f,B(g,h))+B(g,B(h,f))+B(h,B(f,g)))(x,y)\\
		&\quad=\sg(x,y)^2([f,[g,h]]+[g,[h,f]]+[h,[f,g]])(x,y)=0.
		\end{align*}
	\end{proof}
	
	\begin{cor}
		Each $R$-bilinear antisymmetric biderivation of $FI(P,R)$ defines a Poisson structure on $FI(P,R)$.
	\end{cor}
	
	\begin{thrm}\label{Bider-is-Poisson}
		The Poisson structures $B$ on $FI(P,R)$ are in a one-to-one correspondence with the maps $\sg:P^2_<\to R$ which are constant on chains. The correspondence is given by \cref{B(f_g)(x_y)=sg(x_y)[f_g](x_y)}, in which $f,g\in FI(P,R)$.
	\end{thrm}

	\bibliography{bibl}{}
	\bibliographystyle{acm}
	
\end{document}